\pdfoutput=1 %for arXiv pdflatex option
\newif\ifpreprint
\preprintfalse  % enable this if using the publisher's class
\ifpreprint
\documentclass[a4paper,11pt]{article}

\usepackage[utf8]{inputenc}
\usepackage{a4wide}
\usepackage{latexsym,amsfonts,amsmath,amssymb,mathrsfs,url,amsthm}
\providecommand{\keywords}[1]
{
  \noindent \small	
  \textbf{Keywords:} #1
}
\providecommand{\amscode}[1]
{
  \noindent \small	
  \textbf{AMS subject classifications:} #1
}
\else
  \documentclass{amsart}
\fi

\usepackage{mathtools}
\usepackage{todonotes}
\usepackage{hyperref}
\usepackage{color,graphicx}

\newtheorem{theorem}{Theorem}[section]
\newtheorem{lemma}[theorem]{Lemma}

\newtheorem{algorithm}[theorem]{Algorithm}
\newtheorem{remark}[theorem]{Remark}

\newtheorem{definition}[theorem]{Definition}

\newcommand{\Mstar}{M_{\mkern-0.9mu\raisebox{2.2pt}{${}_\star$}}}

\newcommand{\rd}{\, \mathrm{d}}

\newcommand{\EE}{\mathbb{E}}
\newcommand{\NN}{\mathbb{N}}

\newcommand{\RR}{\mathbb{R}}

\newcommand{\ZZ}{\mathbb{Z}}
\newcommand{\Hcal}{\mathcal{H}}
\newcommand{\Ocal}{\mathcal{O}}

\newcommand{\rmse}{\mathrm{rmse}}
\newcommand{\wor}{\mathrm{wor}}

\DeclareMathOperator{\supp}{supp}

\DeclareMathOperator{\ter}{ter}

\usepackage{xcolor}
\definecolor{darkblue}{RGB}{0,60,180} 
\definecolor{darkgreen}{RGB}{0,130,70}
\definecolor{darkorange}{RGB}{180,60,0}
\mathtoolsset{showonlyrefs=true}

\allowdisplaybreaks

\ifpreprint
\title{Randomizing the trapezoidal rule gives\\ the optimal RMSE rate in Gaussian Sobolev spaces\thanks{This work of the third author was supported by NTNU project grant 81617985.}}
\else
\title[Optimal randomized quadrature for Gaussian weight]{Randomizing the trapezoidal rule gives\\ the optimal RMSE rate in Gaussian Sobolev spaces}
\fi

\ifpreprint
\author{Takashi Goda\thanks{School of Engineering, University of Tokyo, 7-3-1 Hongo, Bunkyo-ku, Tokyo 113-8656, Japan.}, Yoshihito Kazashi\thanks{Department of Mathematics and Statistics, University of Strathclyde, Glasgow G1 1XH, UK}, Yuya Suzuki\thanks{Department of Mathematical Sciences, Norwegian University of Science and Technology, Sentralbygg II, Alfred Getz’ vei 1, Gl{\o}shaugen, 7034 Trondheim, Norway}}
\else
\author{Takashi Goda}
\address{School of Engineering, University of Tokyo, 7-3-1 Hongo, Bunkyo-ku, Tokyo 113-8656, Japan.}
\email{goda@frcer.t.u-tokyo.ac.jp}
\author{Yoshihito Kazashi}
\address{Department of Mathematics and Statistics, University of Strathclyde, Glasgow G1 1XH, UK}
\email{y.kazashi@strath.ac.uk}
\author{Yuya Suzuki}
\address{Department of Mathematical Sciences, Norwegian University of Science and Technology, Sentralbygg II, Alfred Getz’ vei 1, Gl{\o}shaugen, 7034 Trondheim, Norway; Department of Mathematics and Systems Analysis, Aalto University School of Science, Espoo, FI-00076 Aalto, Finland}
\email{yuya.suzuki@aalto.fi}

\thanks{This work was supported by JSPS KAKENHI Grant Number 20K03744 (T.G.) and by NTNU project grant 81617985 (Y.S.).}

\fi

\date{\today}

\begin{document}

\maketitle

\begin{abstract}
Randomized quadratures for integrating functions in Sobolev spaces of order $\alpha \ge 1$, where the integrability condition is with respect to the Gaussian measure, are considered. 
In this function space, the optimal rate for the worst-case root-mean-squared error (RMSE) is established. 
Here, optimality is for a general class of quadratures, in which adaptive non-linear algorithms with a possibly varying number of function evaluations are also allowed. 
The optimal rate is given by showing matching bounds. First, a lower bound on the worst-case RMSE of $O(n^{-\alpha-1/2})$ is proven, where $n$ denotes an upper bound on the expected number of function evaluations. 
It turns out that a suitably randomized trapezoidal rule attains this rate, up to a logarithmic factor. 
A practical error estimator for this trapezoidal rule is also presented. Numerical results support our theory.
\end{abstract}

\ifpreprint
\amscode{68W20, 65C05, 65D30, 65D32}
\fi

%%%%%%%%%%%%%%%%%%%%%%%%%%%%%%%%%%%%%%%%%%%%%%%%%%
%%%%%%%%%%%%%%%%%%%%%%%%%%%%%%%%%%%%%%%%%%%%%%%%%%
\section{Introduction}
We study numerical integration of functions with respect to the standard Gaussian measure. Given a function $f: \RR\to \RR$, our problem is to approximate the integral
\[ I(f)=\int_{\RR}f(x)\rho(x)\rd x\quad \text{with}\quad \rho(x)=\frac{1}{\sqrt{2\pi}}e^{-x^2/2},\]
based on pointwise function evaluations.

Of particular interest in this paper are randomized algorithms, which typically combine a fixed (deterministic) quadrature rule and the Monte Carlo sampling; see for example \cite[Section~12.7]{Monahan.JF_2011_NumericalMethodsStatistics} and references therein as well as \cite{Ow97,Di11,Ul17}.
One benefit of these algorithms is that one may construct online error estimators to assess their accuracy, which is often difficult with deterministic algorithms. 
Another is that they are robustly efficient for various smoothness of the integrand. 
Even in the absence of smoothness of the integrand, which is typically required for deterministic quadratures, they may work well at least as accurate as the plain vanilla Monte Carlo method; in the presence of smoothness, they exploit it and may achieve accuracy superior to the Monte Carlo methods.

How accurate can a randomized algorithm be?  
To address this question, we fix a function class and a figure of merit. 
We then consider a general class of algorithms and derive a lower bound for the error, which quantifies the best possible accuracy using this class of algorithms in the sense specified below. 

It turns out that this best possible rate is, up to a logarithmic factor, attained by a randomized trapezoidal rule, which we propose in this paper. 
This result is a random counterpart of our recent result \cite{KSG2022}, in which we showed that the trapezoidal rule is, up to a logarithmic factor, optimal to approximate $I(f)$ in a suitable sense.

Following \cite{KSG2022}, throughout this paper, we assume that $f$ is in a Sobolev space of integer order $\alpha\geq 1$, denoted by $\Hcal_{\alpha}$. 
More precisely, with the normed vector space $L_{\rho}^2(\RR)$ defined by 
\[ \|f\|^2_{L_{\rho}^2} := \int_{\RR}|f(x)|^2\rho(x)\rd x < \infty,\]
for an integer $\alpha\geq 1$,  $\Hcal_{\alpha}$ is defined by 
\[ \Hcal_{\alpha} := \left\{ f\in L_{\rho}^2(\RR)  \;\middle|\; \|f\|_{\alpha}:=\left(\sum_{\tau=0}^{\alpha}\|f^{(\tau)}\|^2_{L_{\rho}^2}\right)^{1/2}<\infty \right\}, \]
where $f^{(\tau)}$ denotes the $\tau$-th weak derivative of $f$.

For this class of functions, we quantify the efficiency of the algorithms in terms of the worst-case root-mean-squared error (RMSE), where the mean is taken on the underlying probability space where the randomization is considered. To define this, we first introduce a class of deterministic algorithms; realizations of the randomized algorithm we consider take their values in the following class of deterministic algorithms.

In the deterministic worst-case setting which we build upon, we approximate $I(f)$ by a (deterministic) quadrature rule of the form 
\begin{align}\label{eq:deterministic_rule}
A^{\mathrm{det}}_{m}(f)=\varphi_{m(f)}\left( x_1,x_2,\ldots,x_{m(f)},f(x_1),f(x_2),\ldots,f(x_{m(f)})\right),
\end{align}
with nodes $x_i\in \RR$, $i=1,\dots,m(f)$, 
where $\varphi_{m(f)}: \RR^{2m(f)}\to \RR$ is a (linear or non-linear) mapping. 
Here, each node $x_i$ can be chosen sequentially by using the information already obtained through $x_1,\ldots,x_{i-1}$ and $f(x_1),\ldots,f(x_{i-1})$, for $i\ge 2$, i.e., with
$x_1$ pre-selected independently of $f$, 
$x_i=\psi_i(x_1,\dots,x_{i-1},f(x_1),\dots,f(x_{i-1}))$ for some (not necessarily linear) function $\psi_i$,  $i\ge 2$. 
This sequential choice of nodes is continued until a stopping criterion is satisfied, and the number of quadrature nodes $m$ is not determined a priori. More precisely, the sequential choice of nodes is terminated 
%if
as soon as
\begin{equation}
\ter_i(x_1,\dots,x_i,f(x_1),\dots,f(x_i))=1,\label{eq:Boolean-func}
\end{equation}
where $\ter_i: \RR^{2i}\to \{0,1\}$ is a Boolean function %$\ter_i: \RR^{2i}\to \{0,1\}$ 
for each $i$. 
Then the positive integer $m(f)$, representing the total number of function evaluations, is given by
\[ m(f)=\min\left\{ i\in \NN \; \middle|\;  \ter_i(x_1,\dots,x_i,f(x_1),\dots,f(x_i))=1\right\}.\]
For example, if we define the functions $\ter_i$ by $\ter_i\equiv 0$ if $i<m$ and $\ter_m\equiv 1$ for some fixed $m\in \NN$, then the number of quadrature nodes is equal to $m$ independently of the integrand $f$. Another example is the case where $\ter_i$ is a function of an error estimator
\[
\mathrm{Estimator}(x_1,\dots,x_i,f(x_1),\dots,f(x_i))
\]
and a user-specified error tolerance $\mathrm{TOL}$; $\ter_i$ is defined so that it returns $1$ if the estimated error is below $\mathrm{TOL}$, and returns $0$ otherwise. In this case, the number of quadrature nodes in general depends on $f$. 

In this paper, following the standard terminology in the context of information-based complexity \cite[Sections~4.1.1 \& 4.2.1]{NW2008} and \cite{W1986,N1996}, we say that this type of sequential node selection is \emph{adaptive} and that \eqref{eq:deterministic_rule} is an \emph{adaptive algorithm}. 
If the number $m(f)$ does not vary with $f$ and each node $x_i$ is chosen without the information $f(x_1),\ldots,f(x_{m})$, we say that \eqref{eq:deterministic_rule} is a \emph{non-adaptive algorithm}. 
In this context, the adaptive stopping rule and the adaptive node selection are usually considered separately, and each type of adaptivity improves an algorithm in different ways; see \cite{W1986,N1996}.

Our interest in this paper is in the randomized setting. We follow 
\cite[Definition~4.35]{NW2008} for the class of algorithms we consider:
\begin{definition}\label{def:rand_alg}
A randomized quadrature rule $A$ is a pair of a probability space $(\Omega,\Sigma, \mu)$ and a family $(A^{\omega})_{\omega\in \Omega}$ of mappings such that the following holds:
\begin{enumerate}
    \item For each fixed $\omega\in\Omega$, the mapping $A^{\omega}: \Hcal_{\alpha}\to \RR$ is a  deterministic quadrature rule of the form~\eqref{eq:deterministic_rule}.
    \item Let $M(f,\omega)$ be the number of nodes used in $A^{\omega}$ for $f\in \Hcal_{\alpha}$ and $\omega$. Then the function $\omega\mapsto M(f,\omega)$ is measurable for each fixed $f$.
\end{enumerate}
\end{definition}
\noindent
Finally, for a randomized quadrature rule $A$, its \emph{cardinality} is defined as the supremum of the expected number of nodes over $f\in \Hcal_{\alpha}$:
\[ \#(A):=\sup_{f\in \Hcal_{\alpha}}\int_{\Omega}M(f,\omega)\rd \mu(\omega).\]

We quantify the efficiency of the algorithm by the worst-case error. 
For a  deterministic algorithm $A^{\mathrm{det}}_m$ as in \eqref{eq:deterministic_rule}, the (deterministic) worst-case error in $\Hcal_{\alpha}$ is given by
\begin{equation}\label{eq:det-WCE}
 e^{\wor}(A^{\mathrm{det}}_m,\Hcal_{\alpha}):=\sup_{\substack{f\in \Hcal_{\alpha}\\ \|f\|_{\alpha}\leq 1}}\left| I(f)-A^{\mathrm{det}}_m(f)\right|.   
\end{equation}
For the randomized setting, we consider the worst-case RMSE
\[ e^{\rmse}(A,\Hcal_{\alpha}) := \sup_{\substack{f\in \Hcal_{\alpha}\\ \|f\|_{\alpha}\leq 1} }\left(\int_{\Omega}(I(f)-A^{\omega}(f))^2  \rd \mu(\omega)\right)^{1/2}.\]
For $A^{\mathrm{det}}_m$ we have  $e^{\rmse}(A^{\mathrm{det}}_m,\Hcal_{\alpha})=e^{\wor}(A^{\mathrm{det}}_m,\Hcal_{\alpha})$.

The main contribution of this paper is two-fold.
\begin{enumerate}
    \item
    We establish a lower bound for   the worst-case RMSE $e^{\rmse}(A,\Hcal_{\alpha})$ for any integer $\alpha\geq1$. More precisely, we prove that for any $n\geq 1$  and any randomized quadrature rule $A$ with cardinality $\#(A)\leq n$, $e^{\rmse}(A,\Hcal_{\alpha})$ is bounded from below by $c_{\alpha}n^{-\alpha-1/2}$ with a constant $c_{\alpha}>0$ depending only on $\alpha$; see  Section~\ref{sec:lower_bound}.
    \item 
    We propose an algorithm that achieves this rate up to a  logarithmic factor; hence, our algorithm is optimal, up to a logarithmic factor, in the sense of the worst-case RMSE. More precisely, we develop a suitably truncated randomized trapezoidal rule with the number of nodes $M(f,\omega)$ being randomly chosen independently of $f\in \Hcal_{\alpha}$. This quadrature rule is shown to be an unbiased estimator of $I(f)$, which allows for practical error estimation.
\end{enumerate}
As briefly mentioned earlier, the second contribution above is motivated by our recent paper \cite{KSG2022}, in which we considered a suitably truncated trapezoidal rule and established the optimality in terms of the worst-case error. 
There, building upon the work of Nuyens and Suzuki \cite{NS2021}, we showed that for any $\alpha\geq 1$  a suitably truncated trapezoidal rule achieves the optimal rate $\Ocal(n^{-\alpha})$ up to a logarithmic factor.

As already mentioned, we randomize this trapezoidal rule and establish the optimality in the sense of the worst-case RMSE. Without randomizing the same algorithm is sub-optimal; it will achieve only $\Ocal(n^{-\alpha})$ in the sense of the worse-case RMSE, while the lower bound we establish in this paper is $n^{-\alpha-1/2}$ up to a constant factor. 
We also note that the worst-case deterministic error for the standard, deterministic, Gauss--Hermite quadrature in $\Hcal_{\alpha}$ is bounded from below and above by $n^{-\alpha/2}$ up to some constant factors (note also that, from $e^{\rmse}(A^{\mathrm{det}}_m,\Hcal_{\alpha})=e^{\wor}(A^{\mathrm{det}}_m,\Hcal_{\alpha})$ for deterministic algorithms, these matching bounds of the rate $n^{-\alpha/2}$ also apply to the worst-case RMSE for the standard deterministic Gauss--Hermite quadrature). 
This rate is merely half of the best possible rate; see \cite{DILP2018} together with \cite{KSG2022}.

Our idea for randomization comes from a classical work of Bakhvalov \cite{Ba1961}, who introduced the method of choosing the number of quadrature nodes randomly. His method has been explored further quite recently in the context of quasi-Monte Carlo integration over the high-dimensional unit cube \cite{DGS2022,KKNU2019}. 
Another related work is by Wu \cite{Wu2022}, who considered a randomized trapezoidal rule similar to ours. Compared to our methods, Wu used a fixed number of nodes; the function space considered there was the fractional Sobolev spaces of order less than $2$ over a finite interval, for which the author showed that the convergence rate in $L^p$ ($p\in[2,\infty)$), where the $p$-integrability is with respect to the underlying probability measure, improves the standard convergence rate by $1/2$. 
In contrast, our results show that, for the function classes we consider, we are able to exploit smoothness more than $2$ and improve the $L^2$-convergence rate by $1/2$. See also Remark~\ref{rem:rate-and-tail} for a related discussion.

We start by showing a general lower bound for the worst-case RMSE in Section~\ref{sec:lower_bound}. 
The lower bound established in this section shows the best possible rate. 
In Section~\ref{sec:random_trap}, we show that this rate, up to a logarithmic factor, is achieved by a randomized trapezoidal rule. Numerical results in Section~\ref{sec:experiments} support our theory. 
%%%%%%%%%%%%%%%%%%%%%%%%%%%%%%%%%%%%%%%%%%%%%%%%%%
%%%%%%%%%%%%%%%%%%%%%%%%%%%%%%%%%%%%%%%%%%%%%%%%%%
%\section{Preliminaries}\label{sec:setting}

\section{A general lower bound}\label{sec:lower_bound}
In this section, we prove the following general lower bound on $e^{\rmse}(A,\Hcal_{\alpha})$ that holds for any randomized quadrature rule $A$ with its cardinality $\#(A)$ at most $n$. 

\begin{theorem}[General lower bound on the worst-case RMSE]\label{thm:lower_bound}
Let integers $\alpha\geq 1$ and $n\geq 1$ be given. For any randomized quadrature rule $A$, possibly adaptive or nonlinear, as in Definition~\ref{def:rand_alg} with $\#(A)\leq n$, the inequality
\[ e^{\rmse}(A,\Hcal_{\alpha})\geq \frac{c_{\alpha}}{n^{\alpha+1/2}} \]
holds, where the constant $c_{\alpha}>0$ depends on $\alpha$ but is independent of $n$ and $A$.
\end{theorem}

To show this result, we first construct fooling functions, which are crucial for our proof. For $n\in \NN$ given, consider $10n$ intervals 
\begin{align}
    \iota_{j}:=(\xi_{j-1},\xi_{j}):=((j-1)/n,j/n),\quad j=-5n+1,\dots,5n.
    \label{eq:intervals}
\end{align}
By construction, $\iota_{j}$, $j=-5n+1,\dots,5n$ are disjoint. For each $j$, we construct a function $f_j\in \Hcal_{\alpha}$ 
supported on $\iota_j$ that integrates to a large value, still having a unit Sobolev norm.

\begin{lemma}[Fooling function and its integral]\label{lem:fooling}
Let integers $\alpha\geq 1$ and $n\geq 1$ be given. Then, for each $j=-5n+1,\dots,5n$, there exists a function $f_j\in \Hcal_{\alpha}$ such that $f_j\geq 0$, $\supp(f_j)=\iota_j$, $\|f_j\|_{\alpha}=1$, and
\[I(f_j)\geq \frac{\eta_{\alpha}}{n^{\alpha+1/2}},\]
where $\eta_{\alpha}>0$ is a constant depending on $\alpha$ but independent of $n$, $j$ and $f_j$. Here, $\supp(f_j)$ denotes the set-theoretical support $\supp(f_j)=\{x\in\mathbb{R}\mid f_j\neq 0\}$.
\end{lemma}

\begin{proof}
Define $h_j: \RR\to \RR$ such that $h_j\geq 0$ and $\supp(h_j)=\iota_j$ by
\[ h_j(x)=\begin{cases} \displaystyle \left(\frac{x-\xi_{j-1}}{\xi_j-\xi_{j-1}}\right)^{\alpha}\left( 1-\frac{x-\xi_{j-1}}{\xi_j-\xi_{j-1}}\right)^{\alpha} & \text{if $x\in \iota_j,$}\\ 0 & \text{otherwise.}\end{cases}\]
As we showed in \cite[Proof of Lemma~3.1]{KSG2022}, we have $h_j\in \Hcal_{\alpha}$. 
Moreover, with
\[ S_{\alpha,\tau} := \sum_{\ell_1,\ell_2=0}^{\alpha}(-1)^{\ell_1+\ell_2}\binom{\alpha}{\ell_1}\binom{\alpha}{\ell_2}\frac{(\alpha+\ell_1)!}{(\alpha+\ell_1-\tau)!}\frac{(\alpha+\ell_2)!}{(\alpha+\ell_2-\tau)!},\]
we have
\begin{align*}
    \|h_j\|_{\alpha}^2 & = \sum_{\tau=0}^{\alpha}\int_{\iota_j}
        \bigl| h_j^{(\tau)}(x)\bigr|^2 \rho(x)\rd x 
    \leq \sum_{\tau=0}^{\alpha}\frac{e^{-\min(\xi_{j-1}^2,\xi_j^2)/2}}{\sqrt{2\pi}}\int_{\iota_j}
        \bigl| h_j^{(\tau)}(x)\bigr|^2 \rd x \\
    & \leq \sum_{\tau=0}^{\alpha}\frac{S_{\alpha,\tau}}{\sqrt{2\pi}(\xi_j-\xi_{j-1})^{2\tau-1}} = \sum_{\tau=0}^{\alpha}\frac{n^{2\tau-1}S_{\alpha,\tau}}{\sqrt{2\pi}} \leq n^{2\alpha-1}\sum_{\tau=0}^{\alpha}\frac{S_{\alpha,\tau}}{\sqrt{2\pi}},
\end{align*}
and
\begin{align*}
    I(h_j) & = \int_{\iota_j} h_j(x)\rho(x)\rd x \geq \frac{e^{-\max(\xi_{j-1}^2,\xi_j^2)/2}}{\sqrt{2\pi}}\int_{\iota_j} h_j(x)\rd x \\
    & \geq \frac{e^{-25/2}}{\sqrt{2\pi}}(\xi_j-\xi_{j-1})\int_0^1x^{\alpha}(1-x)^{\alpha}\rd x = \frac{(\alpha!)^2}{n(2\alpha+1)!\sqrt{2\pi e^{25}}}.
\end{align*}
It follows that, with $f_j:=h_j/\|h_j\|_{\alpha}$, we have $f_j\in \Hcal_{\alpha}$, $f_j\geq 0$, $\supp(f_j)=\iota_j$, $\|f_j\|_{\alpha}=1$
and
\begin{align*} I(f_j)= \frac{I(h_j)}{\|h_j\|_{\alpha}}\geq \frac{(\alpha!)^2}{n^{\alpha+1/2}(2\alpha+1)!\sqrt{2\pi  e^{25}}}\left( \sum_{\tau=0}^{\alpha}\frac{S_{\alpha,\tau}}{\sqrt{2\pi}}\right)^{-1/2}=: \frac{\eta_{\alpha}}{n^{\alpha+1/2}},
\end{align*}
for $j=-5n+1,\ldots,5n$. This completes the proof.
\end{proof}

We are now ready to prove Theorem~\ref{thm:lower_bound}.
\begin{proof}[Proof of Theorem~\ref{thm:lower_bound}]
Let $f_j$, $j=-5n+1,\ldots,5n$, be as in Lemma~\ref{lem:fooling}. By construction, we have $\|f_j\|_{\alpha}=1$. 
Hence, the definition of  $e^{\rmse}(A,\Hcal_{\alpha})$ implies
\begin{align}\label{eq:ave_case}
    & e^{\rmse}(A,\Hcal_{\alpha}) \notag \\
    & \quad \geq \max_{f\in\{\pm f_{j}\mid j=-5n+1,\dots,5n\}}\left(\int_{\Omega}(I(f)-A^{\omega}(f))^2 \rd\mu(\omega)\right)^{1/2} \notag \\
    & \quad \geq \left( \frac{1}{10n}\sum_{j=-5n+1}^{5n}\int_{\Omega}\frac{(I(f_j)-A^{\omega}(f_j))^2+(I(-f_j)-A^{\omega}(-f_j))^2}{2} \rd \mu(\omega)\right)^{1/2}\notag\\
    & \quad = \left(
    \frac{1}{10n}\int_{\Omega}\sum_{j=-5n+1}^{5n}\frac{(I(f_j)-A^{\omega}(f_j))^2+(I(-f_j)-A^{\omega}(-f_j))^2}{2} \rd \mu(\omega)\right)^{1/2}.
\end{align}
We will show that at least $n$ out of $10n$ functions $f_j$, $j=-5n+1,\dots,5n$, satisfy  $(I(f_j)-A^{\omega}(f_j))^2+(I(-f_j)-A^{\omega}(-f_j))^2\geq {2\eta_{\alpha}^2}/{n^{2\alpha+1}}$ on a set of positive measure $E\subset \Omega$. Here, the choice of these $n$ functions may depend on $\omega\in E$.

Define 
$E\in\Sigma$, where $\Sigma$ is the $\sigma$-algebra as in Definition~\ref{def:rand_alg}, by
\[ E:=\left\{\omega\in \Omega\; \middle|\;  \frac{1}{10n}\sum_{j=-5n+1}^{5n}M(f_j,\omega)< 2n\right\}.
\]
Moreover, for $\omega \in E$ denote by $F_0^{\omega}$ the subset of $f_j$'s whose corresponding quadrature $A^{\omega}(f_j)$ uses at most $4n$ nodes:
with \[{\mathrm{Ind}_0^{\omega}:=
\bigl\{
    j\in \{-5n+1,\dots,5n\}
\;\big|\;
M({f_j},\omega)\leq 4n \bigr\}},\]   
let 
\begin{equation*}
F_0^{\omega}:=\bigl\{f_j\in \{f_{-5n+1},\dots,f_{5n}\}\;\big|\; j\in \mathrm{Ind}_0^{\omega}\bigr\}.
\end{equation*}
For $i=1,2,\ldots, 4n,$ we recursively define
\[\mathrm{Ind}_i^{\omega}:=
\bigl\{
    j\in \mathrm{Ind}_{i-1}^{\omega}
\;\big|\;
i\leq M(f_j,\omega),\, 
x^{(j)}_i\notin \iota_j \bigr\}
\cup 
\bigl\{
    j\in \mathrm{Ind}_{i-1}^{\omega}
\;\big|\;
M(f_j,\omega)<i \bigr\},
\]
where $x^{(j)}_i=x^{(j)}_i(\omega)$ denotes the $i$-th node of $A^{\omega}(f_j)$ and $\iota_{j}$ is defined in \eqref{eq:intervals}, 
and 
\[ F_{i}^{\omega}:=\bigl\{ f_j\in F_{i-1}^{\omega}\mid j\in \mathrm{Ind}_i^{\omega}
\bigr\}.
\]
$F_{i}^{\omega}$ is constructed so that if $i\leq M(f_j,\omega)$, i.e., when the $i$-th node  exists, then $f_j\in F_{i}^{\omega}$ implies  $f_j(x_i^{(j)})=0$, and if $M(f_j,\omega)<i$, i.e., when the $i$-th node does not exist, then we keep $f_j\in F^{\omega} _{i-1}$ in  $F^{\omega} _i$. 
We will show $|F_{4n}^{\omega}|\geq n$ for $\omega\in E$ and that $F_{4n}^{\omega}$ is the sought set that gives the aforementioned lower bound.

First, we show that $E$ has a positive measure. 
By the definition of the cardinality of a randomized quadrature rule, it holds that
\begin{align*}
 \int_{\Omega}\frac{1}{10n}\sum_{j=-5n+1}^{5n}M(f_j,\omega)\rd \mu(\omega) &\leq \sup_{j=-5n+1,\ldots,5n}\int_{\Omega}M(f_j,\omega) \rd \mu(\omega) 
 \\
 & \leq \sup_{f\in \Hcal_{\alpha}}\int_{\Omega}M(f,\omega)\rd \mu(\omega)\leq n.
 \end{align*}
With $E^c:=\Omega\setminus E$, it follows from Markov's inequality that
\[ \mu(E)=1-\mu(E^c)\geq 1-\frac{1}{2n}\int_{\Omega}\frac{1}{10n}\sum_{j=-5n+1}^{5n}M(f_j,\omega)\rd \mu(\omega)\geq \frac{1}{2}.\]

For $\omega\in E$ it turns out  $|F_0^{\omega}|\geq 5n$, i.e., 
%for at least half of $10n$ functions $f_j$'s,  $j=-5n+1,\ldots,5n$, 
for at least half of the $10n$ functions $f_j$, $j=-5n+1,\ldots,5n$, 
the integral estimate $A^{\omega}(f_j)$ is computed by no more than $4n$ nodes.
Indeed, from Markov's inequality for the uniform measure $\nu$ on $(\{-5n+1,\dots,5n\},2^{\{-5n+1,\dots,5n\}})$ and the corresponding expectation $\mathbb{E}_\nu$, together with the definition of $E$ we see that the number of  $f_j$'s whose corresponding nodes exceed $4n$ is less than half:
\begin{align*}
    \frac{|\{j\in\{-5n+1,\ldots,5n\}\mid M(f_j,\omega)> 4n\}|}{10n}
    & = \nu(\{j\mid M(f_j,\omega)> 4n\}) \\
    & \kern-2cm \leq \frac1{4n}\mathbb{E}_\nu[M(f_{\cdot},\omega)]=
    \frac{1}{4n}\cdot \frac{1}{10n}\sum_{j=-5n+1}^{5n}M(f_j,\omega)<\frac{1}{2}.
\end{align*}

Next, we will show $| F_{i}^{\omega}|\geq  |F_{i-1}^{\omega}|-1$ for $i=1,\dots,4n$. 
For this, we will show that, given $\omega\in E$, the first $i$ quadrature nodes of $A^\omega(f_j)$ for $f_j\in F^\omega_{i}$ are uniquely determined, independently of $f_j\in F^\omega_{i}$.

To see this, first fix $f_j\in 
F_{i}^{\omega}\subset 
F_{i-1}^{\omega}$ arbitrarily. By construction of $F_{i}^{\omega}$, the interval $\iota_j$ does not contain any of the first $i-1$ nodes  $x_1^{(j)},\dots,x_{i-1}^{(j)}$ of $A^{\omega}(f_j)$. Although the next node $x_i^{(j)}$ is (possibly adaptively) determined by the set of $x^{(j)}_1,\ldots,x^{(j)}_{i-1}$ and $f_j(x^{(j)}_1),\ldots,f_j(x^{(j)}_{i-1})$, these function values are all equal to $0$. 
Now recall that the first node $x^{(j)}_1=x_1$ of $A^\omega(f_j)$  does not depend on $f_j$, since initially we do not have any information on $f_j$. 
Therefore, by the recursive definition of $F_i^{\omega}$, we see that for $f_j\in F_i^{\omega}$ the first $i$ quadrature nodes of $A^\omega(f_j)$ are all equal for $f_j\in F_i^{\omega}$, independently of the corresponding index $j$. 
Indeed, for $f_j,f_{j'}\in F_i^{\omega}$, we have
\begin{align*}
    x^{(j)}_{\ell}
    &=
    \psi_\ell(x_1^{(j)},x_2^{(j)},\dots,x_{\ell-1}^{(j)},f_j(x_1^{(j)}),\dots,f_j(x_{\ell-1}^{{(j)}}))\\
    &=
    \psi_\ell(x_1,x_2^{(j)},\dots,x_{\ell-1}^{(j)},0,\dots,0)\\
    &=
    \psi_\ell(x_1,x_2^{(j')},\dots,x_{\ell-1}^{(j')},0,\dots,0)=x^{(j')}_{\ell}\quad\text{ for}\quad\ell=2,\dots, 
    \min\{i,M(f_{j},\omega)\},
\end{align*}
and if $\min\{i,M(f_{j},\omega)\}=M(f_{j},\omega)$ then 
for $\ell= M(f_{j},\omega)+1,\dots,i$ the point $x^{(j)}_{\ell}$ does not exist.
Here, note that if $M(f_{j},\omega)<i$ then $M(f_{j},\omega)=M(f_{j'},\omega)$, and if $i\leq M(f_{j},\omega)$ then $i\leq M(f_{j'},\omega)$. 
Indeed, for $\ell=1,\dots,i$, the Boolean function \eqref{eq:Boolean-func} satisfies
\begin{align*}
    & \ter_{\ell}(x_1^{(j')},x_2^{(j')},\dots,x_{\ell}^{(j')},f_{j'}(x_1^{(j')}),\dots,f_{j'}(x_{\ell}^{{(j')}})) \\
    & \quad =
    \ter_{\ell}(x_1^{(j)},x_2^{(j)},\dots,x_{\ell}^{(j)},0,\dots,0)\\
    & \quad =
    \ter_{\ell}(x_1^{(j)},x_2^{(j)},\dots,x_{\ell}^{(j)},f_{j}(x_1^{(j)}),\dots,f_{j}(x_{\ell}^{{(j)}})).
\end{align*}
Now, for $f_j\in F_{i-1}^{\omega}$,  we have $f_j\not\in F_{i}^{\omega}$ if and only if $x_i^{(j)}\in \iota_j$. 
Since the $i$-th node $x^{(j)}_i$ can be contained in at most one interval $\iota_{\bar{j}}=((\bar{j}-1)/n,\bar{j}/n)$ among the supports $(\iota_{k})_k$ corresponding to  $(f_k)_k\subset F_{i-1}^{\omega}$, 
the above argument leads to 
$|\{f_j\in F_{i-1}^{\omega}\mid f_j\not\in F_{i}^{\omega}\}|\leq 1$, and thus
\[ | F_{i}^{\omega}|\geq  |F_{i-1}^{\omega}|-1.\]
Therefore, $|F_{4n}^{\omega}|\geq |F_0^{\omega}|-4n\geq n$ holds. 

Now, from $F_{4n}^{\omega}\subset F_{0}^{\omega}$, for $f_j\in F_{4n}^{\omega}$ no quadrature point of $A^{\omega}(f_j)$ is in the support of $f_j$. 
Therefore, $|F_{4n}^{\omega}|\geq n$ 
means that among $10n$ functions  $f_{j}$, $j=-5n+1,\dots,5n$, at least $n$ of them have no quadrature node of $\{A^{\omega}(f_j)\}_{j=-5n+1,\dots,5n}$ in their supports. 
Hence, for these (at least) $n$ functions, the function values used in $A^{\omega}(f_j)$ are all $0$, which also holds true for $-f_j$. Therefore, for these $f_j$'s, we obtain the equality $A^{\omega}(f_j)=A^{\omega}(0)=A^{\omega}(-f_j)$ and
\begin{align*}
& (I(f_j)-A^{\omega}(f_j))^2+(I(-f_j)-A^{\omega}(-f_j))^2 \\
& \quad =(I(f_j)-A^{\omega}(f_j))^2+(-I(f_j)-A^{\omega}(f_j))^2 \notag \\
& \quad = 2(I(f_j))^2+2(A^{\omega}(f_j))^2\geq 2(I(f_j))^2\geq \frac{2\eta_{\alpha}^2}{n^{2\alpha+1}}.\label{eq:error-fj-nice}
\end{align*}
For the rest of the (at most $4n$) functions, $(I(f_j)-A^{\omega}(f_j))^2+(I(-f_j)-A^{\omega}(-f_j))^2$ is trivially bounded below by $0$. 

Applying the results to \eqref{eq:ave_case}, we obtain 
\[ e^{\rmse}(A,\Hcal_{\alpha}) \geq \left(\mu(E)\cdot \left(\frac{n}{10n}\cdot \frac{\eta_{\alpha}^2}{n^{2\alpha+1}}+\frac{9n}{10n}\cdot 0\right)\right)^{1/2} \ge \frac{\eta_{\alpha}}{2\sqrt{5}n^{\alpha+1/2}},  \]
which completes the proof of the theorem.
%%%%%%%%%%%%%%%%%%%%
\end{proof}

\begin{remark}
The idea of proving a lower bound on the randomized error by the average-case error over bump functions with mutually disjoint supports goes back to Bakhvalov \cite{Ba1959}. We also refer to \cite[Chapter~2]{No1988} and \cite[Chapter~17]{NW2010} for more information on error analyses for the randomized setting.
\end{remark}

%%%%%%%%%%%%%%%%%%%%%%%%%%%%%%%%%%%%%%%%%%%%%%%%%%
%%%%%%%%%%%%%%%%%%%%%%%%%%%%%%%%%%%%%%%%%%%%%%%%%%
\section{Randomized trapezoidal rule}\label{sec:random_trap}
Here we develop a randomized quadrature rule $A$ that achieves the almost optimal rate of convergence in the sense of the worst-case RMSE $e^{\rmse}(A, \Hcal_{\alpha})$. This algorithm is based on the suitably truncated deterministic trapezoidal rule we considered in \cite{KSG2022}, which is given by
\[ A^*_{n,T}=\frac{2T}{n}\sum_{j=0}^{n-1}f(\xi_j^*)\rho(\xi_j^*) \quad \text{with}\quad \xi_j^*=T\left(\frac{2j}{n}-1\right), j=0,\ldots,n-1.\]
Here, $n$ denotes the number of nodes and $T>0$ is a parameter that controls the cut-off of the integration domain from $\RR$ to $[-T,T]$. As shown in \cite{KSG2022}, $A^*_{n,T}$ turns out to achieve, up to a logarithmic factor, the optimal rate of convergence in the sense of the worst-case error in $\Hcal_{\alpha}$, for any integer $\alpha\geq 1$. 

In this paper, we introduce its randomized counterpart that makes $e^{\rmse}(A,\Hcal_{\alpha})$ small. In what follows, we denote the cumulative distribution function of the standard Gaussian distribution by $\Phi$.

\begin{algorithm}\label{alg:rand_trap}
    Let $f:\RR\to \RR$, $n\geq 4$ and $T>0$ be given. To approximate the integral $I(f)$, 
    with independent random variables $\Mstar\sim U\{\lfloor n/2\rfloor, \ldots, n-2\}$ and $\delta\sim U(0,1)$, 
    we define the randomized trapezoidal rule $A_{n,T}=(A_{n,T}^{\Mstar,\delta})_{\Mstar,\delta}$ by
    \begin{align}\label{eq:rand_trap}
        A_{n,T}^{\Mstar,\delta}=\frac{2T}{\Mstar}\sum_{j=0}^{\Mstar-1}f(\xi_j^*)\rho(\xi_j^*)+\Phi(-T)f(\xi_{\mathrm{left}}^*)+(1-\Phi(T))f(\xi_{\mathrm{right}}^*),
    \end{align}
where we let
\[ \xi_j^{*}:=T\left(\frac{2(j+\delta)}{\Mstar}-1\right),\quad j=0,\ldots,\Mstar-1,\]
and the end nodes $\xi_{\mathrm{left}}^*$ and $\xi_{\mathrm{right}}^*$ are independent of $\Mstar$ and sampled from the truncated normal distributions on $(-\infty,-T]$ and $[T,\infty)$, respectively.
\end{algorithm}

Each node $\xi_j$ in the 
first term of \eqref{eq:rand_trap} is  uniformly distributed over the interval $(T(2j/\Mstar-1), T(2(j+1)/\Mstar-1))$, $j=0,\dots,\Mstar-1$. 
For $\xi_{\mathrm{left}}^*$ and $\xi_{\mathrm{right}}^*$, we specify only their marginal distributions, truncated normals. 
Any random variables $\xi_{\mathrm{left}}^*$ and $\xi_{\mathrm{right}}^*$ having the marginals specified above that are independent of $\Mstar$, in particular those that are not independent of $\xi_j$'s, allow the results in this paper to hold.
For instance, one straightforward implementation is to use the inversion method, i.e., $\xi_{\mathrm{left}}^*:=\Phi^{-1}(\delta \Phi(-T))$, and $\xi_{\mathrm{right}}^*:= \Phi^{-1}((1-\delta) \Phi(T)+\delta)$, where $\delta$ is as in Algorithm~\ref{alg:rand_trap}. 
As such, we omit $\xi_{\mathrm{left}}^*$ and $\xi_{\mathrm{right}}^*$ from the notation $A_{n,T}^{\Mstar,\delta}$.

However, in view of implementation, sampling from such truncated distributions by the inversion method can be numerically unstable and inefficient when $T$ is large; see \cite{Chopin.N_2011_FastSimulationTruncated}. As pointed out therein, one needs to be careful when using the inversion method for the right tail of the distribution. Hence, it can be more convenient to sample them independently by other methods.
We refer to \cite[Chapter~9.1]{De86} and \cite{Botev.ZI_2017_NormalLawLinear,Chopin.N_2011_FastSimulationTruncated} among many others for random number generation from the tail of normal distributions.

In Algorithm~\ref{alg:rand_trap}, the cut-off parameter $T>0$ should be carefully chosen depending on $n$ (and $\alpha$). This point will be discussed later in this section. The symmetry of the standard Gaussian distribution around $0$ implies  $\Phi(-T)=1-\Phi(T)$ for any $T>0$, i.e., the function values at $\xi_{\mathrm{left}}^*$ and $\xi_{\mathrm{right}}^*$ are weighted equally in $A_{n,T}^{\Mstar,\delta}$.

We emphasize that $A_{n,T}$ is a randomized non-adaptive linear quadrature rule; $\Mstar$, $\delta$, $\xi_{\mathrm{left}}^*$ and $\xi_{\mathrm{right}}^*$ are the only $\omega$-dependent components of $A_{n,T}$ and neither depends on the integrand $f$.
Furthermore, we have $\#(A_{n,T})\leq n$. 
This is because $\Mstar\sim U\{\lfloor n/2\rfloor, \ldots, n-2\}$, and moreover, for each realization of $\Mstar$ and $\delta$, the deterministic rule $A_{n,T}^{\Mstar,\delta}$ utilizes $\Mstar+2$ nodes:  $\xi_{\mathrm{left}}^*$, $\xi_{\mathrm{right}}^*$, and $\xi_j^{*}$, $j=0,\ldots,\Mstar-1$. 
Hence, our result in Theorem~\ref{thm:upper_bound} below shows that adaptivity and nonlinearity are not needed to achieve the optimality (up to a logarithmic factor) in the sense that is considered here (cf.~the general lower bound in Theorem~\ref{thm:lower_bound}).

%%%%%%%%%%%%%%%%%%%%%%%%%%%%%%%%%%%%%%%%%%%%%%%%%%
\subsection{Unbiasedness and error estimation}\label{subsec:unbiased}
Before proving an upper bound on $e^{\rmse}(A_{n,T},\Hcal_{\alpha})$,  we show that our randomized quadrature rule is unbiased and allows for practical error estimation. 
Let us start by showing the unbiasedness of $A_{n,T}$.

\begin{lemma}[Unbiasedness]\label{lem:unbiased}
Let $\alpha\geq 1$, $n\geq 4$, and $T>0$ be given. For any $f\in \Hcal_{\alpha}$, we have
\[ \EE\left[ A_{n,T}^{\Mstar,\delta}(f)\right] := \frac{1}{n-1-\lfloor n/2\rfloor}\sum_{m_\star=\lfloor n/2\rfloor}^{n-2}\int_0^1 A_{n,T}^{m_\star,\delta}(f) \rd \delta = I(f). \]
\end{lemma}

\begin{proof}
Substituting \eqref{eq:rand_trap} into $\EE\left[ A_{n,T}^{\Mstar,\delta}(f)\right]$, we have
\begin{align*}
    \EE\left[ A_{n,T}^{\Mstar,\delta}(f)\right] & = \EE\left[ \frac{2T}{\Mstar}\sum_{j=0}^{\Mstar-1}f(\xi_j^{*})\rho(\xi_j^{*})+\Phi(-T)f(\xi_{\mathrm{left}}^*)+(1-\Phi(T))f(\xi_{\mathrm{right}}^*)\right] \\
    & = \EE\left[\frac{2T}{\Mstar}\sum_{j=0}^{\Mstar-1} f(\xi_j^{*})\rho(\xi_j^{*})\right]\\
    & \quad +\Phi(-T)\EE\left[ f(\xi_{\mathrm{left}}^*)\right]+(1-\Phi(T))\EE\left[ f(\xi_{\mathrm{right}}^*)\right].
\end{align*}

For the first term on the rightmost side above, we have
\begin{align*}
    & \EE\left[\frac{2T}{\Mstar}\sum_{j=0}^{\Mstar-1} f(\xi_j^{*})\rho(\xi_j^{*})\right] \\
    & \quad = \frac{1}{n-1-\lfloor n/2\rfloor}\sum_{m_\star=\lfloor n/2\rfloor}^{n-2}\int_0^1 \left(\frac{2T}{m_\star}\sum_{j=0}^{m_\star-1} f(\xi_j^{*})\rho(\xi_j^{*})\right) \rd \delta \\
    & \quad = \frac{1}{n-1-\lfloor n/2\rfloor}\sum_{m_\star=\lfloor n/2\rfloor}^{n-2}\frac{2T}{m_\star}\sum_{j=0}^{m_\star-1} \int_0^1 f(\xi_j^{*})\rho(\xi_j^{*}) \rd \delta\\
    & \quad = \frac{1}{n-1-\lfloor n/2\rfloor}\sum_{m_\star=\lfloor n/2\rfloor}^{n-2}\sum_{j=0}^{m_\star-1} \int_{T(2j/m_\star-1)}^{T(2(j+1)/m_\star-1)} f(x)\rho(x) \rd x\\
    & \quad = \frac{1}{n-1-\lfloor n/2\rfloor}\sum_{m_\star=\lfloor n/2\rfloor}^{n-2} \int_{-T}^{T} f(x)\rho(x) \rd x = \int_{-T}^{T} f(x)\rho(x) \rd x,
\end{align*}
where the third equality follows from the change of variables $x=T(2(j+\delta)/m_\star-1)$ for each $j=0,\ldots,m_\star-1$. Regarding the second term, since $\xi_{\mathrm{left}}^*$ is independent of $\Mstar$, we have
\begin{align*}
    \Phi(-T)\EE\left[ f(\xi_{\mathrm{left}}^*)\right] = \int_{-\infty}^{-T}f(x)\rho(x)\rd x.
\end{align*}
In a similar way, for the third term it can be shown that
\[ (1-\Phi(T))\EE\left[ f(\xi_{\mathrm{right}}^*)\right]=\int_{T}^{\infty}f(x)\rho(x)\rd x.\]

Altogether we obtain
\begin{align*}
\EE\left[ A_{n,T}^{\Mstar,\delta}(f)\right]&=\int_{-T}^{T} f(x)\rho(x) \rd x+\int_{-\infty}^{-T}f(x)\rho(x)\rd x+\int_{T}^{\infty}f(x)\rho(x)\rd x
\\
& =\int_{-\infty}^{\infty}f(x)\rho(x)\rd x=I(f),
\end{align*}
which proves the statement.
\end{proof}

In practice, we use the approximation 
\[
\overline{A_{n,T}(f)}:=\frac{1}{r}\sum_{i=1}^{r}A_{n,T}^{\Mstar^{(i)},\delta^{(i)}}(f) \approx I(f),
\]
with $r\in\mathbb{N}$ independent random
variables
\begin{align*} 
& (\Mstar^{(1)},\delta^{(1)},\xi_{\mathrm{left}}^{*,(1)},\xi_{\mathrm{right}}^{*,(1)}),\dots,(\Mstar^{(r)},\delta^{(r)},\xi_{\mathrm{left}}^{*,(r)},\xi_{\mathrm{right}}^{*,(r)})
\\
& \qquad \qquad \in \{\lfloor n/2\rfloor, \ldots, n-2\} \times (0,1)\times (-\infty,-T]\times [T,\infty), 
\end{align*}
which can be easily generated. Lemma~\ref{lem:unbiased} yields an error bound and a practical error estimator for this algorithm.

From Lemma~\ref{lem:unbiased}, $A_{n,T}^{\Mstar^{(i)},\delta^{(i)}}(f)$, $i=1,\ldots,r$, are independent and unbiased estimators of $I(f)$.  Thus $\overline{A_{n,T}(f)}$ is also an unbiased estimator of $I(f)$, and thus
\begin{align}\label{eq:practical-error}
    \EE\left[ \left( \overline{A_{n,T}(f)}-I(f)\right)^2\right] = \frac{1}{r}\EE\left[ \left( A_{n,T}^{\Mstar,\delta}(f)-I(f)\right)^2\right],%\leq %\frac{1}{r}C^2_{\alpha,\lambda}\|f\|_{\alpha}^2\frac{(\ln n)^{\alpha+1/2}}{n^{2\alpha+1}}.
\end{align}
and 
\begin{align}\label{eq:whole-MSE=Var}
    \EE\left[ \left( \overline{A_{n,T}(f)}-I(f)\right)^2\right] = \EE\left[ \left( \overline{A_{n,T}(f)}-\EE\left[\overline{A_{n,T}(f)}\right]\right)^2\right].
\end{align}

As we shall show in the proof of Theorem~\ref{thm:upper_bound}, with a suitable $T>0$, for any $f\in \Hcal_{\alpha}$ the mean-squared error can be bounded as 
$\EE[( A_{n,T}^{\Mstar,\delta}(f)-I(f))^2]\leq C^2_{\alpha,\lambda}\|f\|_{\alpha}^2\frac{(\ln n)^{\alpha+1/2}}{n^{2\alpha+1}}$, 
where $C_{\alpha,\lambda}>0$ is a constant depending on $\alpha$ and $\lambda$ but independent of $f$, and thus \eqref{eq:practical-error} implies the error bound
\[
\EE\left[ \left( \overline{A_{n,T}(f)}-I(f)\right)^2\right]\leq C^2_{\alpha,\lambda}\|f\|_{\alpha}^2\frac{1}{r}
\frac{(\ln n)^{\alpha+1/2}}{n^{2\alpha+1}}.
\]

From \eqref{eq:whole-MSE=Var}, the sample variance
\begin{align}\label{eq:sample_variance}
    \frac{1}{r(r-1)}\sum_{i=1}^{r}\left( A_{n,T}^{\Mstar^{(i)},\delta^{(i)}}(f)-\overline{A_{n,T}(f)}\right)^2
\end{align}
can be used as an online, unbiased mean-squared error estimator.

%%%%%%%%%%%%%%%%%%%%%%%%%%%%%%%%%%%%%%%%%%%%%%%%%%
\subsection{A bound on the root-mean-square error}
As the second main result of this paper, we show an upper bound for the worst-case RMSE $e^{\rmse}(A_{n,T},\Hcal_{\alpha})$:
\begin{theorem}[Upper bound on the worst-case RMSE]\label{thm:upper_bound}
Let $\alpha\geq 1$ and $n\geq 4$ be given. 
Fix $\lambda\in (1/2,1)$ arbitrarily. 
Then, with
\begin{align}\label{eq:cutoff_choice}
T = \sqrt{\frac{2\alpha+1}{1-\lambda}\ln (n)},
\end{align}
$A_{n,T}$ as in Algorithm~\ref{alg:rand_trap} satisfies
\[ e^{\rmse}(A_{n,T},\Hcal_{\alpha})\leq C_{\alpha,\lambda}\frac{(\ln n)^{\alpha/2+1/4}}{n^{\alpha+1/2}},\]
where $C_{\alpha,\lambda}>0$ is a constant depending only on $\alpha$ and $\lambda$.
\end{theorem}
This rate is optimal up to a logarithmic factor in the sense of the worst-case RMSE; indeed, Theorem~\ref{thm:lower_bound} shows that the optimal rate is $\Ocal(n^{-\alpha-1/2})$. To prove this theorem, first we need some preparations.

\begin{remark}
    In Theorem~\ref{thm:upper_bound}, the %choice of 
    parameter $\lambda\in(1/2,1)$ does not affect the convergence rate, including the logarithmic factor. However, it does change the constant $C_{\alpha,\lambda}$.  
    This change is caused by changes of (i) the choice of the cut-off parameter $T$ and (ii) the constant $C_{\alpha,\lambda}^*$, both of which depend on $\lambda$; see \eqref{eq:explicitconstant1} and \eqref{eq:explicitconstant2} below.
    For increasing $\lambda$, $T$ increases but $C_{\alpha,\lambda}^*$  decreases.  
    It is not very straightforward to find the optimal $\lambda$ minimizing the constant, and this is out of the scope of the present paper.
\end{remark}

\subsubsection{Auxiliary results}
Because our algorithm is based on the trapezoidal rule yet the integrand is not assumed to be periodic, error analysis requires extra work.

First, we introduce the following notations 
for an $\alpha$-times weakly differentiable function $F\colon \mathbb{R}\to \mathbb{R}$. 
Let $F^{(\tau)}$ denote its $\tau$-th weak derivative for $\tau=0,\dots,\alpha$. Define 
\[ \|F\|_{\alpha}^* := \sup_{\substack{I\subset \RR\\ |I|<\infty}}\|F\|_{\alpha,I}^* \]
with
\[ \|F\|_{\alpha,I}^*:=\left( \sum_{\tau=0}^{\alpha-1}\left(\int_{I}F^{(\tau)}(x)\rd x \right)^2+\int_{I}|F^{(\alpha)}(x)|^2\rd x\right)^{1/2}\]
and for $\lambda\in (1/2,1)$, 
\[ \|F\|_{\alpha,\lambda,\mathrm{decay}} := \sup_{x\in \RR, \tau\in \{0,\ldots,\alpha-1\}}\left| e^{(1-\lambda)x^2/2}F^{(\tau)}(x)\right|,\]
allowing this quantity to be infinite. It turns out that $f\in \Hcal_{\alpha}$ implies $ \|f\rho\|_{\alpha}^*<\infty$, and $\|f\rho\|_{\alpha,\lambda,\mathrm{decay}}<\infty$, as we showed in \cite[Proof of Theorem~4.5]{KSG2022}.

\begin{lemma}[Norm inequality]\label{lem:KSG}
Let $\alpha\geq 1$ and $\lambda\in (1/2,1)$ be given. For $f\in \Hcal_{\alpha}$, let $F(x):=f(x)\rho(x)$. Then, we have
\[ \|F\|_{\alpha}^*+\|F\|_{\alpha,\lambda,\mathrm{decay}}\leq C^*_{\alpha,\lambda}\|f\|_{\alpha} \]
with a constant $C^*_{\alpha,\lambda}>0$ depending on $\alpha$ and $\lambda$ but independent of $f$.
\end{lemma}

Moreover, following \cite{KSG2022} we use the following periodization \eqref{eq:periodize} as an auxiliary variable. We refer to \cite[Lemma~4.1]{KSG2022} for a proof.
\begin{lemma}[Auxiliary periodic function and its properties]\label{lem:periodize}
Let $\alpha\geq 1$ and $T>0$ be given. For $f\in \Hcal_{\alpha}$, let $F(x):=f(x)\rho(x)$ and define $G: [-T-d, T+d]\to \RR$ by
\begin{align}\label{eq:periodize}
    G(x):=F(x)-\sum_{\tau=1}^{\alpha}\frac{B^{[-T,T]}_{\tau}(x)}{\tau!}\left( \int_{-T}^{T}F^{(\tau)}(y)\rd y\right),
\end{align}
for $x\in [-T-d,T+d],$ with an arbitrarily small $d>0$, where $B^{[-T,T]}_{\tau}$ denotes the scaled Bernoulli polynomial of degree $\tau$ on $[-T,T]$, i.e.,
\[ B^{[-T,T]}_{\tau}(x)=(2T)^{\tau-1}B_{\tau}\left( \frac{x+T}{2T}\right)\]
with $B_{\tau}$ being the standard Bernoulli polynomial of degree $\tau$. Then, the following holds:
\begin{enumerate}
    \item The function $G$ preserves the integral of $F$ on $[-T,T]$:
    \[\int_{-T}^{T}G(x)\rd x = \int_{-T}^{T}F(x)\rd x = \int_{-T}^{T}f(x)\rho(x)\rd x.\]
    \item The function $G$ is $(\alpha-1)$-times continuously differentiable on $(-T-d,T+d)$ with $G^{(\alpha-1)}$ being absolutely continuous on $[-T,T]$, and satisfies $G^{(\tau)}(-T)=G^{(\tau)}(T)$ for all $\tau=0,\ldots,\alpha-1$.
    \item The norm of $G$ is bounded above by that of $F$:
    \begin{equation}\label{eq:GFnormbound} \|G\|_{\alpha,[-T,T]}^*\leq \|F\|_{\alpha,[-T,T]}^* \;,
    \end{equation}
    where we note that $\|G\|_{\alpha,[-T,T]}^*$ is well defined, since the continuous differentiability and the absolute continuity above implies the weak differentiability of $G^{(\tau)}$ on $[-T,T]$ for all $\tau\leq \alpha-1$.
\end{enumerate}
\end{lemma}
The functions defined above will be used to analyze the integration error on $[-T,T]$.
In the proof, we decompose the integrand $f(x)\rho(x)$ into two parts; the periodic part which is represented by $G(x)$, and the non-periodic part. We use $\|f\rho\|_{\alpha}^*$ to control the integration error of the periodic part, while the error of the non-periodic part is controlled by $\|f\rho\|_{\alpha,\lambda,\mathrm{decay}}$, as can be seen in the proof of Lemma~\ref{lem:proof_1st-term} below. The cut-off parameter $T$ in \eqref{eq:cutoff_choice} is given in such a way that these two errors are balanced.
To control the error on the tail regions $(-\infty,-T]$ and $[T,\infty)$,  we only need $\|f\|_{L_\rho ^2}$, as shown in the proof of Lemma~\ref{lem:proof_other-terms} below.  
\begin{remark}
The method of periodizing a function by adding Bernoulli polynomials in \eqref{eq:periodize} was first introduced by Korobov~\cite{Ko1963}, and also studied by Zaremba~\cite{Za1972}, in the context of numerical integration. This method was originally proposed for numerically integrating non-periodic functions on a multidimensional box. This periodization is referred to as Bernoulli polynomial method in \cite[Section~2.12]{SJ1994}. Recently, a new proof strategy was developed for a class of quasi-Monte Carlo methods in \cite{NS2021}, where this periodized function was used as an auxiliary function only appearing in a proof, but not in the algorithm. In particular, \cite{NS2021} showed that $G(x)$ is an orthogonal projection of $F(x)$ in a suitable sense, from one kind of Sobolev space to a corresponding periodic Sobolev space, in a multidimensional setting. This implies the norm inequality shown in Lemma~\ref{lem:periodize}.3. In \cite{KSG2022} the authors applied this proof strategy for the Gaussian Sobolev spaces.     
\end{remark}

\subsubsection{Proof of Theorem~\ref{thm:upper_bound}}

Let us prove Theorem \ref{thm:upper_bound}. It follows from the proof of Lemma~\ref{lem:unbiased} and Jensen's inequality that, for any $f\in \Hcal_{\alpha}$, the mean-squared error of $A_{n,T}(f)$ is bounded above as
\begin{align}
    & \EE\left[ \left( A_{n,T}^{\Mstar,\delta}(f)-I(f)\right)^2\right] \notag \\
    & \quad = \EE\left[ \left(\frac{2T}{\Mstar}\sum_{j=0}^{\Mstar-1}f(\xi_j^*)\rho(\xi_j^*)+\Phi(-T)f(\xi_{\mathrm{left}}^*)+(1-\Phi(T))f(\xi_{\mathrm{right}}^*) \right.\right. \notag \\
    & \quad \qquad \qquad \left.\left.-\int_{-T}^{T}f(x)\rho(x)\rd x-\int_{-\infty}^{-T}f(x)\rho(x)\rd x-\int_{T}^{\infty}f(x)\rho(x)\rd x\right)^2\right] \notag \\
    & \quad \leq 3\, \EE\left[ \left(\frac{2T}{\Mstar}\sum_{j=0}^{\Mstar-1}f(\xi_j^*)\rho(\xi_j^*)-\int_{-T}^{T}f(x)\rho(x)\rd x\right)^2\right]  \notag \\ 
    & \quad \quad + 3\, \EE\left[ \left(\Phi(-T)f(\xi_{\mathrm{left}}^*)-\int_{-\infty}^{-T}f(x)\rho(x)\rd x\right)^2\right] \label{eq:whole-decomp} \\
    & \quad \quad +3\, \EE\left[ \left((1-\Phi(T))f(\xi_{\mathrm{right}}^*)-\int_{T}^{\infty}f(x)\rho(x)\rd x\right)^2\right]\notag.
\end{align}
Thus it suffices to give an upper bound on each term of \eqref{eq:whole-decomp}. In Lemma~\ref{lem:proof_1st-term} below, we show an upper bound on the first term. And then in Lemma~\ref{lem:proof_other-terms}, we show upper bounds on the second and third terms.

\begin{lemma}[{Upper bound on the RMSE for $[-T,T]$}]\label{lem:proof_1st-term}
Let $\alpha\geq 1$, $n\geq 4$, and $T>0$ be given.
Fix $\lambda\in (1/2,1)$ arbitrarily. 
For any $f\in \Hcal_{\alpha}$, the first term of \eqref{eq:whole-decomp} is bounded as
\begin{align*}
    & \EE\left[ \left(\frac{2T}{\Mstar}\sum_{j=0}^{\Mstar-1}f(\xi_j^*)\rho(\xi_j^*)-\int_{-T}^{T}f(x)\rho(x)\rd x\right)^2\right] \\
    & \quad \leq  (C^*_{\alpha,\lambda})^2\|f\|_{\alpha}^2\left(  \frac{2\alpha^2\max\{1, (2T)^{2\alpha}\}}{e^{(1-\lambda)T^2}}+ \frac{2^{2\alpha+5}T^{2\alpha+1}}{\pi^{2\alpha}n^{2\alpha+1}}\right),
\end{align*}
with $C^*_{\alpha,\lambda}>0$ being the same constant as  in Lemma~\ref{lem:KSG}.
\end{lemma}

\begin{proof}
Let us consider the auxiliary periodic function $G: [-T,T]\to \RR$ defined in \eqref{eq:periodize}. 
From Lemma~\ref{lem:periodize}.1, we have
\begin{align}
    & \EE\left[ \left(\frac{2T}{\Mstar}\sum_{j=0}^{\Mstar-1}f(\xi_j^*)\rho(\xi_j^*)-\int_{-T}^{T}f(x)\rho(x)\rd x\right)^2\right] \notag \\
    & = \EE\left[ \left(\frac{2T}{\Mstar}\sum_{j=0}^{\Mstar-1}\left(F(\xi_j^*)-G(\xi_j^*)\right)+\frac{2T}{\Mstar}\sum_{j=0}^{\Mstar-1}G(\xi_j^*)-\int_{-T}^{T}G(x)\rd x\right)^2\right] \notag \\
    & \leq 2\,\EE\left[ \left(\frac{2T}{\Mstar}\sum_{j=0}^{\Mstar-1}\left(F(\xi_j^*)-G(\xi_j^*)\right)\right)^2\right] \notag\\
    & \quad +2\,\EE\left[\left(\frac{2T}{\Mstar}\sum_{j=0}^{\Mstar-1}G(\xi_j^*)-\int_{-T}^{T}G(x)\rd x\right)^2\right]. \label{eq:1st-term_bound}
\end{align}

For the first term of \eqref{eq:1st-term_bound}, it follows from the periodization \eqref{eq:periodize} that, for both $\Mstar$ and $\delta$ given,
\begin{align}
        & \left(\frac{2T}{\Mstar}\sum_{j=0}^{\Mstar-1}\left(F(\xi_j^*)-G(\xi_j^*)\right)\right)^2 \notag \\
        & \quad = 4T^2\left(\frac{1}{\Mstar}\sum_{j=0}^{\Mstar-1}\sum_{\tau=1}^{\alpha}\frac{B^{[-T,T]}_{\tau}(\xi_j^*)}{\tau!}\left( \int_{-T}^{T}F^{(\tau)}(y)\rd y\right)\right)^2 \notag \\
        & \quad \leq 4T^2\left(\sum_{\tau=1}^{\alpha}\frac{(2T)^{\tau-1}}{2}\left| F^{(\tau-1)}(T)-F^{(\tau-1)}(-T) \right|\right)^2 \notag \\
        & \quad \leq \left(\alpha \max\{1, (2T)^{\alpha}\}\|F\|_{\alpha,\lambda,\mathrm{decay}} \; e^{-(1-\lambda)T^2/2}\right)^2 \notag \\
        & \quad \leq \alpha^2(C^*_{\alpha,\lambda})^2\|f\|_{\alpha}^2 \max\{1, (2T)^{2\alpha}\} \; e^{-(1-\lambda)T^2},\label{eq:bound_on_1st_1st_term}
    \end{align}
where we used $|B^{[-T,T]}_{\tau}(x)/\tau!|\leq (2T)^{\tau-1}/2$ for $x\in [-T,T]$, shown in \cite{L1940}, to get the first inequality, and applied Lemma~\ref{lem:KSG} in the last inequality. Since the bound  \eqref{eq:bound_on_1st_1st_term} does not depend on the realization of $\Mstar(\omega)$ or on that of $\delta(\omega)$, the first term of \eqref{eq:1st-term_bound} is bounded above by this bound multiplied by $2$.
        
Let us move on to the second term of \eqref{eq:1st-term_bound}. Lemma~\ref{lem:periodize}.2 implies that $G$ has the pointwise-convergent Fourier series
\[ G(x)=\sum_{k\in \ZZ}\widehat{G}(k)\phi_k^{[-T,T]}(x),\]
where $\{\phi_k^{[-T,T]}(x):=\exp(2\pi i k(x+T)/(2T))/\sqrt{2T}\mid k\in \ZZ\}$ forms an orthonormal $L^2([-T,T])$ basis and $\widehat{G}(k)$ denotes the $k$-th Fourier coefficient
\[ \widehat{G}(k)=\int_{-T}^{T}G(x)\overline{\phi_k^{[-T,T]}(x)}\rd x.\]
Besides, it follows from the square integrability of $G^{(\alpha)}$ that $G^{(\alpha)}$ has the $L^2$-convergent Fourier series
\[ G^{(\alpha)}(x)=\sum_{k\in \ZZ}\widehat{G^{(\alpha)}}(k)\phi_k^{[-T,T]}(x)=\sum_{k\in \ZZ}\left( \frac{2\pi i k}{2T}\right)^{\alpha}\widehat{G}(k)\phi_k^{[-T,T]}(x),\]
see \cite[Proof of Lemma~4.1]{KSG2022}. 
Therefore, regarding the second term of \eqref{eq:1st-term_bound}, for $\Mstar$ and $\delta$ both given,
        \begin{align}
            \frac{2T}{\Mstar}\sum_{j=0}^{\Mstar-1}G(\xi_j^*)-\int_{-T}^{T}G(x)\rd x & = \frac{2T}{\Mstar}\sum_{j=0}^{\Mstar-1}\sum_{k\in \ZZ}\widehat{G}(k)\phi_{k}^{[-T,T]}(\xi_j^*)-\sqrt{2T}\widehat{G}(0)\\
            & = \sum_{k\in \ZZ}\widehat{G}(k)\frac{2T}{\Mstar}\sum_{j=0}^{\Mstar-1}\phi_{k}^{[-T,T]}(\xi_j^*)-\sqrt{2T}\widehat{G}(0),
        \end{align}
holds, in which, for any $k\in \ZZ$, we have
        \begin{align}
            \frac{2T}{\Mstar}\sum_{j=0}^{\Mstar-1}\phi_k^{[-T,T]}(\xi_j^*) & = \frac{\sqrt{2T}}{\Mstar}\sum_{j=0}^{\Mstar-1}e^{2\pi ik(\xi_j^*+T)/(2T)} = \frac{\sqrt{2T}}{\Mstar}\sum_{j=0}^{\Mstar-1}e^{2\pi ik(j+\delta)/\Mstar} \\
            & = \frac{e^{2\pi ik \delta/\Mstar}\sqrt{2T}}{\Mstar}\sum_{j=0}^{\Mstar-1}e^{2\pi ik(j/\Mstar)}\\
            & = \begin{cases} e^{2\pi ik \delta/\Mstar}\sqrt{2T} & \text{if $k\equiv 0 \pmod {\Mstar}$,}\\ 0 & \text{otherwise.}\end{cases}\label{eq:character}
        \end{align}
Therefore, we obtain
        \begin{align*}
            & \EE\left[\left(\frac{2T}{\Mstar}\sum_{j=0}^{\Mstar-1}G(\xi_j^*)-\int_{-T}^{T}G(x)\rd x\right)^2\right] \\
            & \quad = \EE\left[\left( \sum_{\substack{k\in \ZZ\setminus \{0\}\\ k\equiv 0 \pmod {\Mstar}}}e^{2\pi ik \delta/\Mstar}\sqrt{2T}\widehat{G}(k)\right)^2\right] \\
            & \quad 
            = 
            2T\, \EE\left[ \sum_{\substack{k,\ell\in \ZZ\setminus \{0\}\\ k, \ell\equiv 0 \pmod {\Mstar}}}e^{2\pi i(k-\ell) \delta/\Mstar}\widehat{G}(k)\overline{\widehat{G}(\ell)}\right]\\
            & \quad 
            = 
            \frac{2T}{n-1-\lfloor n/2\rfloor}\sum_{m_\star=\lfloor n/2\rfloor}^{n-2}\sum_{\substack{k,\ell\in \ZZ\setminus \{0\}\\ k,\ell\equiv 0 \pmod {m_\star}}}\widehat{G}(k)\overline{\widehat{G}(\ell)}\int_0^1 e^{2\pi i(k-\ell) \delta/m_\star}\rd \delta\\
            & \quad = \frac{2T}{n-1-\lfloor n/2\rfloor}\sum_{m_\star=\lfloor n/2\rfloor}^{n-2}\sum_{\substack{k\in \ZZ\setminus \{0\}\\ k\equiv 0 \pmod {m_\star}}}|\widehat{G}(k)|^2 \\
            & \quad = \frac{2T}{n-1-\lfloor n/2\rfloor}\sum_{k\in \ZZ\setminus \{0\}}|\widehat{G}(k)|^2\sum_{m_\star=\lfloor n/2\rfloor}^{n-2}\chi_{m_\star \mid k},
        \end{align*}
where, in the last line, $\chi_{m_\star \mid k}$ is equal to $1$ if $m_\star$ divides $k$ and is equal to $0$ otherwise. 
The last sum over $m_\star$ counts the number of divisors 
of $k$ in $\{\lfloor n/2\rfloor,\ldots,n-2\}$, which we denote by $\tau_{\lfloor n/2\rfloor}^{n-2}(k)$. Then, H\"{o}lder's inequality leads to
        \begin{align*}
            & \EE\left[\left(\frac{2T}{\Mstar}\sum_{j=0}^{\Mstar-1}G(\xi_j^*)-\int_{-T}^{T}G(x)\rd x\right)^2\right]\\
            & = \frac{2T}{n-1-\lfloor n/2\rfloor}\sum_{k\in \ZZ\setminus \{0\}}|\widehat{G}(k)|^2\tau_{\lfloor n/2\rfloor}^{n-2}(k) \\
            & = \frac{2T}{n-1-\lfloor n/2\rfloor}\sum_{k\in \ZZ\setminus \{0\}}|\widehat{G}(k)|^2\left( \frac{2\pi k}{2T}\right)^{2\alpha}\left( \frac{2T}{2\pi k}\right)^{2\alpha}\tau_{\lfloor n/2\rfloor}^{n-2}(k) \\
            & \leq \frac{2T}{n-1-\lfloor n/2\rfloor}\left(\sum_{k\in \ZZ\setminus \{0\}}|\widehat{G}(k)|^2\left( \frac{2\pi k}{2T}\right)^{2\alpha}\right) \sup_{k\in \ZZ\setminus \{0\}}\left( \frac{2T}{2\pi k}\right)^{2\alpha}\tau_{\lfloor n/2\rfloor}^{n-2}(k) \\
            & \leq \frac{2T^{2\alpha+1}}{\pi^{2\alpha}(n-1-\lfloor n/2\rfloor)}\|G\|^2_{\alpha,[-T,T]} \sup_{k\in \ZZ\setminus \{0\}} \frac{\tau_{\lfloor n/2\rfloor}^{n-2}(k)}{k^{2\alpha}},
        \end{align*}
where we used Parseval's identity for the last inequality; see \cite[Proof of Lemma~4.1]{KSG2022}. 
We trivially have $\tau_{\lfloor n/2\rfloor}^{n-2}(k)=0$ for any $|k|< \lfloor n/2\rfloor$. 
For the case $|k|\geq \lfloor n/2\rfloor$, since $m_\star$ divides $k$ if and only if $k/m_\star$ is an integer, 
$\tau_{\lfloor n/2\rfloor}^{n-2}(k)$ is bounded above by the number of integers contained in the interval 
\[ \left[ \frac{|k|}{n-2}, \frac{|k|}{\lfloor n/2\rfloor}\right].\]
Therefore, we have
       \begin{align*}
            \tau_{\lfloor n/2\rfloor}^{n-2}(k) \leq \frac{|k|}{\lfloor n/2\rfloor}-\frac{|k|}{n-2}+1\leq \frac{|k|}{n/2-1}-\frac{|k|}{n-2}+1=\frac{|k|}{n-2}+1.
        \end{align*}
Using this bound, noting $\alpha\geq1$, the supremum over $k$ above is bounded by
        \begin{align*}
            \sup_{k\in \ZZ\setminus \{0\}} \frac{\tau_{\lfloor n/2\rfloor}^{n-2}(k)}{k^{2\alpha}} \leq \sup_{k\geq \lfloor n/2\rfloor} \frac{1}{k^{2\alpha}}\left(\frac{k}{n-2}+1\right)=\frac{1}{\lfloor n/2\rfloor^{2\alpha}}\left(\frac{\lfloor n/2\rfloor}{n-2}+1\right) \leq \frac{2^{2\alpha+1}}{n^{2\alpha}}, 
        \end{align*}
which, together with Lemma~\ref{lem:periodize}.3 and Lemma~\ref{lem:KSG}, yields
        \begin{align*}
            \EE\left[\left(\frac{2T}{\Mstar}\sum_{j=0}^{\Mstar-1}G(\xi_j^*)-\int_{-T}^{T}G(x)\rd x\right)^2\right] & \leq \|G\|^2_{\alpha,[-T,T]}\frac{2^{2\alpha+2}T^{2\alpha+1}}{\pi^{2\alpha}(n-1-\lfloor n/2\rfloor)n^{2\alpha}} \\
            & \leq (C^*_{\alpha,\lambda})^2\|f\|^2_{\alpha} \frac{2^{2\alpha+4}T^{2\alpha+1}}{\pi^{2\alpha}n^{2\alpha+1}}.
        \end{align*}
Now that we have obtained upper bounds on the two terms of \eqref{eq:1st-term_bound}, we complete the proof.
\end{proof}
The benefit of choosing the number of nodes $\Mstar$ randomly manifests itself in the proof of Lemma~\ref{lem:proof_1st-term}. 
As indicated in \eqref{eq:character}, the trapezoidal rule with a fixed number of nodes $m$ does not give exact values for integrating Fourier modes $\phi^{[-T,T]}_k=\exp(2\pi i k(x+T)/(2T))/\sqrt{2T}$, $k\not=0$, if $k$ is a multiple of $m$. Thus, if the number of nodes is fixed, due to the amplitude of the $k$-th Fourier coefficient of the periodized function $G$ with $k$ being a multiple of $m$, the best RMSE attainable is already of order $m^{-\alpha}$ and not less (cf. the proof of \cite[Lemma~4.1]{KSG2022}). 
Although this rate is optimal in the sense of the (deterministic) worst-case error \eqref{eq:det-WCE} (cf.~\cite{KSG2022}), in view of the lower bound in Theorem~\ref{thm:lower_bound} we need a half rate extra for $A_{n,T}$ to be optimal in the sense of the worst-case RMSE.
By choosing $\Mstar$ randomly, on average we can avoid the situation where $k$ is a multiple of $m$, which leads to an improved, optimal rate of the RMSE. 

Not all quadratures require randomization of the number of nodes to achieve the optimality.
Indeed, using (scaled) higher-order scrambled digital nets from \cite{Di11} for the nodes $\xi_j^*$, instead of the equispaced points as in \eqref{eq:rand_trap}, following a similar argument as above, the same rate (up to a logarithmic factor) turns out to be achievable. 
Note however that constructing higher-order nets requires the smoothness parameter $\alpha$ as an input to attain the optimal rate of the RMSE for smooth non-periodic functions, which makes it hard for the resulting randomized algorithm to be free from $\alpha$ in the construction. Algorithm~\ref{alg:rand_trap} compares favorably in this regard, as we discuss later in Remark~\ref{rem:cutoff}.
\begin{lemma}[Upper bound on the RMSE for tails]\label{lem:proof_other-terms}
Let $\alpha\geq 1$, $n\geq 4$, and $T>0$ be given.
For any $f\in \Hcal_{\alpha}$, the second and third terms of \eqref{eq:whole-decomp} are bounded as
\begin{align*}
    \EE\left[ \left(\Phi(-T)f(\xi_{\mathrm{left}}^*)-\int_{-\infty}^{-T}f(x)\rho(x)\rd x\right)^2\right] 
    \leq \|f\|_{\alpha}^2\frac{1}{\sqrt{2\pi}Te^{T^2/2}},
\end{align*}
and
\begin{align*}
    \EE\left[ \left((1-\Phi(T))f(\xi_{\mathrm{right}}^*)-\int_{T}^{\infty}f(x)\rho(x)\rd x\right)^2\right]
    \leq \|f\|_{\alpha}^2\frac{1}{\sqrt{2\pi}Te^{T^2/2}},
\end{align*}
respectively.
\end{lemma}

\begin{proof}
We give a proof only for the bound on the second term of \eqref{eq:1st-term_bound} since the bound on the third term can be proven in the same way. 

From the proof of Lemma~\ref{lem:unbiased}
\[ \EE\left[ \Phi(-T)f(\xi_{\mathrm{left}}^*)\right]=\int_{-\infty}^{-T}f(x)\rho(x)\rd x, \]
holds, and thus
\begin{align*}
    & \EE\left[ \left(\Phi(-T)f(\xi_{\mathrm{left}}^*)-\int_{-\infty}^{-T}f(x)\rho(x)\rd x\right)^2\right] \\
    & \quad = \EE\left[ \left(\Phi(-T)f(\xi_{\mathrm{left}}^*)\right)^2\right]-\left( \int_{-\infty}^{-T}f(x)\rho(x)\rd x\right)^2 \\
    & \quad \leq \EE\left[ \left(\Phi(-T)f(\xi_{\mathrm{left}}^*)\right)^2\right] = \Phi(-T)\int_{-\infty}^{-T}|f(x)|^2\rho(x)\rd x \\
    & \quad \leq \Phi(-T)\|f\|^2_{L^2_\rho}\leq \Phi(-T)\|f\|^2_{\alpha},
\end{align*}
in which we have
\begin{align*}
    \Phi(-T) & = \frac{1}{\sqrt{2\pi}}\int_{-\infty}^{-T}e^{-x^2/2}\rd x = \frac{1}{\sqrt{2\pi}}\int_{T}^{\infty}e^{-x^2/2}\rd x \\
    & \leq \frac{1}{\sqrt{2\pi}}\int_{T}^{\infty}\frac{x}{T}\cdot e^{-x^2/2}\rd x = \frac{e^{-T^2/2}}{T\sqrt{2\pi}}.
\end{align*}
This proves the result.
\end{proof}

Applying the results from Lemma~\ref{lem:proof_1st-term} and \ref{lem:proof_other-terms} to \eqref{eq:whole-decomp}, we obtain
\begin{align*} 
& \EE\left[ \left( A_{n,T}^{\Mstar,\delta}(f)-I(f)\right)^2\right] \\
& \quad \leq 3(C^*_{\alpha,\lambda})^2\|f\|_{\alpha}^2\left(  \frac{2\alpha^2\max\{1, (2T)^{2\alpha}\}}{e^{(1-\lambda)T^2}}+ \frac{2^{2\alpha+5}T^{2\alpha+1}}{\pi^{2\alpha}n^{2\alpha+1}}\right) + \|f\|_{\alpha}^2\frac{6}{\sqrt{2\pi}Te^{T^2/2}},
\end{align*}
so that 
\begin{align*}
    & e^{\rmse}(A_{n,T},\Hcal_{\alpha}) = \sup_{\substack{f\in \Hcal_{\alpha}\\ \|f\|_{\alpha}\leq 1}}\left(\EE\left[ \left( A_{n,T}^{\Mstar,\delta}(f)-I(f)\right)^2\right]\right)^{1/2} \\
    & \quad \leq \left( 3(C^*_{\alpha,\lambda})^2\left(  \frac{2\alpha^2\max\{1, (2T)^{2\alpha}\}}{e^{(1-\lambda)T^2}}+ \frac{2^{2\alpha+5}T^{2\alpha+1}}{\pi^{2\alpha}n^{2\alpha+1}}\right) + \frac{6}{\sqrt{2\pi}Te^{T^2/2}}\right)^{1/2}.
\end{align*}
Choosing the cut-off parameter $T$ according to \eqref{eq:cutoff_choice}, we have $T>1$ and
\begin{align}
    &\quad \quad e^{\rmse}(A_{n,T},\Hcal_{\alpha}) \leq \frac{T^{\alpha+1/2}}{n^{\alpha+1/2}}\left( 3\left(  2^{2\alpha+1}\alpha^2+ \pi^{-2\alpha}2^{2\alpha+5}\right)(C^*_{\alpha,\lambda})^2 + \frac{6}{\sqrt{2\pi}}\right)^{1/2}\label{eq:explicitconstant1}\\
    & \quad = \frac{(\ln n)^{\alpha/2+1/4}}{n^{\alpha+1/2}}\sqrt{\left(\frac{2\alpha+1}{1-\lambda}\right)^{\alpha+1}\left( 3\left(  2^{2\alpha+1}\alpha^2+ \pi^{-2\alpha}2^{2\alpha+5}\right)(C^*_{\alpha,\lambda})^2 + \frac{6}{\sqrt{2\pi}}\right)},\label{eq:explicitconstant2}
\end{align}
which completes the proof of Theorem~\ref{thm:upper_bound}.

\begin{remark}\label{rem:rate-and-tail}
The quadrature for the tail 
$\Phi(-T)f(\xi_{\mathrm{left}}^*)+(1-\Phi(T))f(\xi_{\mathrm{right}}^*)$ in the algorithm \eqref{eq:rand_trap} are utilized only to make the estimator unbiased and thus to obtain the empirical error estimator~\eqref{eq:sample_variance}.  
Indeed, the rate $O(n^{-\alpha-1/2}(\ln n)^{\alpha/2+1/4})$ as in Theorem~\ref{thm:upper_bound} can be established even without these nodes, since under the decay of the integrands assumed, the dominant error comes from the integration error on $[-T,T]$ as in Lemma~\ref{lem:proof_1st-term}.
\end{remark}

\begin{remark}\label{rem:cutoff}
In Theorem~\ref{thm:upper_bound}, the cut-off parameter $T$ depends on the smoothness parameter $\alpha$. This may be a problem in practice, since the smoothness of the target integrand may be unknown. One way to make the algorithm independent of $\alpha$, as we did in \cite[Corollary~4.4]{KSG2022}, is to replace $\alpha$ in \eqref{eq:cutoff_choice} with any slowly increasing function $\gamma: \NN\to [0,\infty)$, such as $\gamma(n)=\ln(n)$ or $\gamma(n)=\max\{\ln(\ln(n)),0\}$. The resulting randomized trapezoidal rule does not require any information about $\alpha$, but still achieves the optimal rate up to a factor of $(\gamma(n)\ln n)^{\alpha/2+1/4}$.
\end{remark}

%%%%%%%%%%%%%%%%%%%%%%%%%%%%%%%%%%%%%%%%%%%%%%%%%%
%%%%%%%%%%%%%%%%%%%%%%%%%%%%%%%%%%%%%%%%%%%%%%%%%%
\section{Numerical experiments}\label{sec:experiments}
We conclude this paper with numerical experiments for our randomized trapezoidal rule $A_{n,T}$. In what follows, all computations are carried out in double precision arithmetic using MATLAB 2020a. Let us consider the following three test functions with different smoothness:
\begin{align*}
    f_1(x) & = f_{1,p}(x) = (\max\{x,0\})^p,\quad \text{with $p  =1,2,3$,}\\
    f_2(x) & = \begin{cases}
    e^{-1/(1-x^2)} & \text{for $|x|<1,$}\\ 0 & \text{otherwise,}\end{cases}\\
    f_3(x) & = (\tanh(x))^2.
\end{align*}
The function $f_1=f_{1,p}$ satisfies $f_1\in \Hcal_{p}$ and $f_1\notin \Hcal_{p+1}$, which we exploit to check how well our theory explains the error that the estimator estimates.
The functions $f_2$ and $f_3$ go beyond our theory, in that the Sobolev class does not capture their smoothness; they are infinitely differentiable and bounded on $\RR$ and thus are in $\Hcal_{\alpha}$ for any $\alpha\in\mathbb{N}$. 
Note that $f_3$ is analytic, while $f_2$ is not.

Throughout the following experiments, we set $\lambda=0.51$. 
We estimate the mean-squared error by following \eqref{eq:sample_variance} with $r=50$ and $n$ being powers of $2$. The reason why we estimate the mean-squared error instead of the RMSE is only because of the fact that the square root of \eqref{eq:sample_variance} is not an unbiased estimator of the RMSE. We also stress that, even though $f_2$ and $f_3$ go beyond our theory, the practical error estimator \eqref{eq:sample_variance} is still valid. 

For $f_1$, we test two cut-off strategies to choose the value of $T$.
One way is choosing $T$ depending on $\alpha$ and we set $\alpha=p$ for $f_1=f_{1,p}$; see \eqref{eq:cutoff_choice} in  Theorem~\ref{thm:upper_bound}.
Another way is to use our $\alpha$-free cut-off, in which we replace  $\alpha$ in \eqref{eq:cutoff_choice} with $\gamma(n)=\max\{\ln(\ln(n)),0\}$; see  Remark~\ref{rem:cutoff}. Figure~\ref{fig:finite} presents error decays for $f_1$ with varying values of $p$, with two cut-off strategies outlined above.   
As expected from the theory, the convergence rate of $n^{-2p-1}$ is achieved for each value of $p$. Moreover, we see that our $\alpha$-free choice of the cut-off parameter does not deteriorate the performance of the randomized trapezoidal rule, which supports our theory.

\begin{figure}
\centering
 \includegraphics[width=0.45\textwidth]{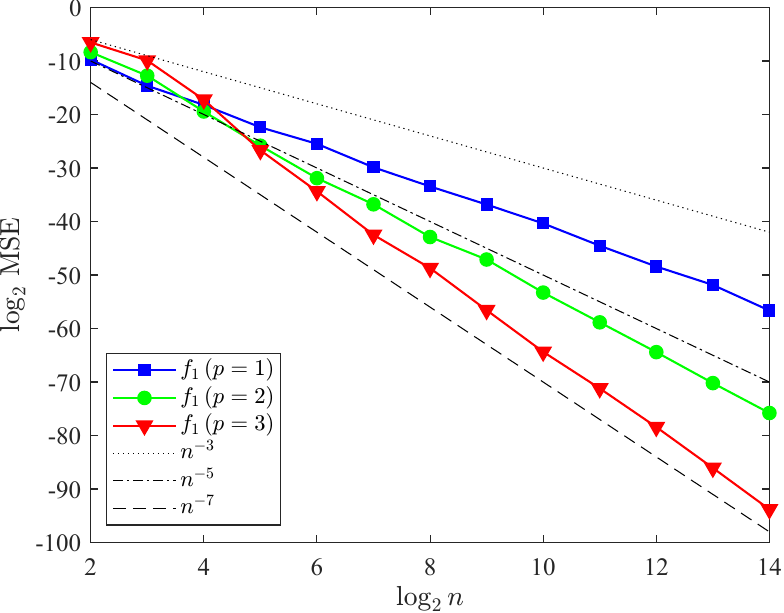}\, \, 
 \includegraphics[width=0.45\textwidth]{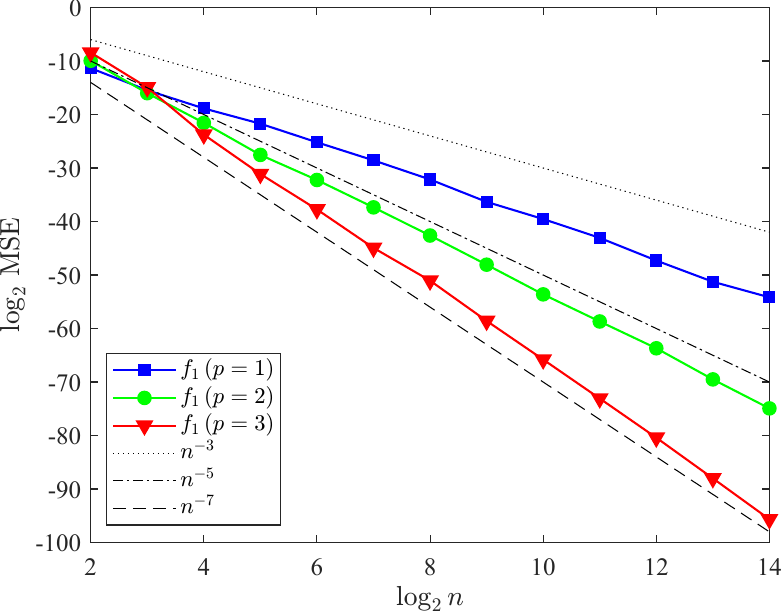}
 \caption{Mean-squared error for $f_1$ with various values of $p$. In the left panel, the cut-off parameter $T$ is chosen as in \eqref{eq:cutoff_choice}; in the right panel,  $T$ is chosen independently of $\alpha$.}
 \label{fig:finite}
\end{figure}

For the functions $f_2$ and $f_3$, we conduct the experiments only for the $\alpha$-free choice of the cut-off parameter $T$. 
Figure~\ref{fig:infinite} shows the results. 
Since the functions are infinitely smooth, the mean-squared error is {expected to} decay at least at the rate $n^{-2\alpha-1}$ for any finite $\alpha$. In fact, we observe that the error decays super-algebraically fast until it drops down to around $2^{-100}$. 
This error decay is consistent with our theory. Nevertheless, it does indicate that our theory may not be suitable for providing sharp error estimates for infinitely smooth functions. For such functions, it may be useful to employ function classes used to analyze e.g., analytic functions.
We refer to \cite{Su1997,TW2014} for error estimates of a deterministic trapezoidal rule for analytic functions. Setting the theory aside, we again stress that the crucial advantage of our randomized rule is that it allows for an error estimation, which is quite hard for deterministic quadrature rules.

\begin{figure}
\centering
 \includegraphics[width=0.45\textwidth]{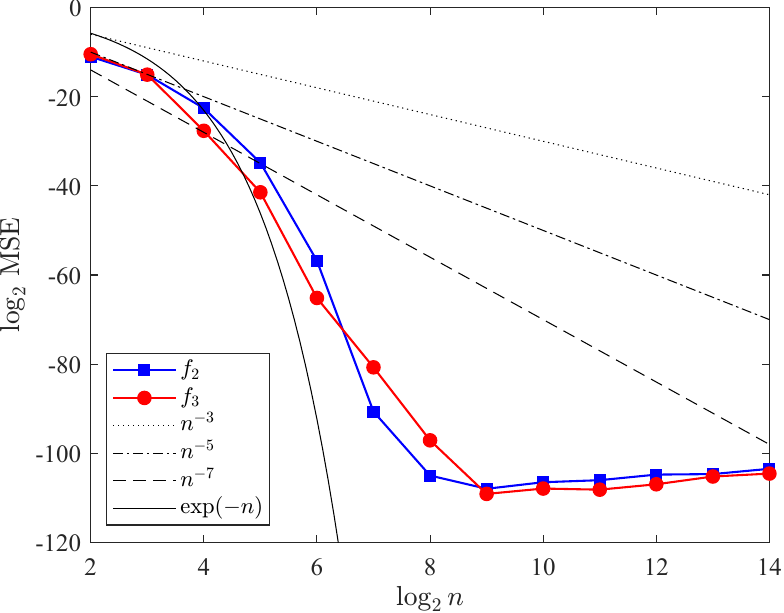}
 \caption{Mean-squared error for $f_2$ and $f_3$.}
 \label{fig:infinite}
\end{figure}

\bibliographystyle{plain}
\bibliography{ref.bib}

\end{document}